\documentclass{amsart}
\usepackage{amsfonts,amssymb,amsmath,enumerate}
\usepackage[all]{xy}
\usepackage{color}
\usepackage{hyperref}

\SelectTips{eu}{}
\xyoption{curve}

\newcommand{\lra}{\longrightarrow}

\def\ls{\leqslant}
\def\gs{\geqslant}

\newcommand{\dtensor}[2]{\otimes^{\!\mathsf L}_{#1}{#2}}

\newcommand{\shift}{{\sf\Sigma}}

\newcommand{\hh}[1]{\operatorname{H}(#1)}
\newcommand{\HH}[2]{\operatorname{H}_{#1}(#2)}

\newcommand\BB{\operatorname{B}}

\newcommand\ZZ{\operatorname{Z}}

\newcommand{\Tor}[4]{\operatorname{Tor}_{#1}^{#2}(#3,#4)}

\newcommand\Ker{\operatorname{Ker}}
\newcommand\Image{\operatorname{Im}}

\newcommand\edim{\operatorname{edim}}
\newcommand\depth{\operatorname{depth}}

\newcommand\im{\operatorname{Im}}

\newcommand\pd{\operatorname{proj\,dim}}

\newcommand\fm{{\mathfrak m}}
\newcommand\fn{{\mathfrak n}}

\newcommand\p{\mathfrak{p}}

\newcommand\BZ{{\mathbb Z}}

\newcommand\PP{\mathcal{P}}

\theoremstyle{plain}
\newtheorem{theorem}{Theorem}[section]

{\Alph{theorem}}
\newtheorem{proposition}[theorem]{Proposition}
\newtheorem{lemma}[theorem]{Lemma}

\newtheorem{corollary}[theorem]{Corollary}

\theoremstyle{definition}

\newtheorem{chunk}[theorem]{}

\theoremstyle{remark}

\newtheorem{remark}[theorem]{Remark}
\numberwithin{equation}{theorem}

\newcommand{\level}[2]{\operatorname{level}^{#1}\!(#2)}
\newcommand{\altlevel}[3]{\operatorname{level}_{#1}^{#2}(#3)}

\newcommand{\hht}[1]{\operatorname{height}{#1}}
\newcommand{\superheight}[1]{\operatorname{superheight}{#1}}

\newcommand{\Spec}[1]{\operatorname{Spec}{#1}}

\begin{document}
%\title[Perfect complexes]
%{Perfect complexes over commutative rings}

\title{Lower bounds on projective levels of complexes}

%[Bounding the  level of perfect complexes]

\date{\today}

\author[H. Altmann]{Hannah Altmann}
\address{Hannah Altmann \\Department of Mathematics and Computer Science\\Bemidji State University\\ 1500 Birchmont Drive NE Box 23\\ Bemidji, MN 56601}
\email{haltmann@bemidjistate.edu}

\author[E. Grifo]{Elo\'{i}sa Grifo}
\address{Elo\'{i}sa Grifo \\ Department of Mathematics \\  University of Virginia \\141 Cabell Drive, Kerchof Hall \\ Charlottesville, VA 22904, USA}

\email{er2eq@virginia.edu}

\author[J. Monta\~{n}o]{Jonathan Monta\~{n}o}
\address{Jonathan Monta\~{n}o \\ Department of Mathematics \\University of Kansas \\405 Snow Hall, 1460 Jayhawk Blvd \\Lawrence, KS 66045}
\email{jmontano@ku.edu}

\author[W. T. Sanders]{William Sanders}  
\address{William T. Sanders \\ Department of Mathematical Sciences \\ Norwegian University of Science and Technology, NTNU\\NO-7491 Trondheim, Norway}
\email{william.sanders@math.ntnu.no}

\author[T. Vu]{Thanh Vu}
\address{Thanh Vu \\ Department of Mathematics, University of Nebraska-Lincoln\\203 Avery Hall\\ Lincoln, NE 68588}
\email{tvu@unl.edu}

%\thanks { We should thank AMS and the MRC program.}

  \begin{abstract}
For an associative ring $R$, the projective level of a complex  $F$  is the smallest number of mapping cones needed to build $F$ from projective $R$-modules. We establish lower bounds for the projective level of $F$ in terms of the vanishing of homology of $F$.  We then use these bounds to derive a new version of The New Intersection Theorem for level when $R$ is a commutative Noetherian local ring. 
  \end{abstract}


\maketitle

\section*{Introduction}
Let $R$ be an associative ring and $F$ a complex of left $R$-modules. The projective level of $F$, denoted by $\level{\PP}{F}$, is the smallest number of steps needed to assemble some projective resolution of $F$ from complexes of projective modules that have zero differentials (see Section \ref{sec:Pre} for details).  Projective level is a special case of a general notion of level in triangulated categories, which was introduced and studied by Avramov, Buchweitz, Iyengar, and Miller in \cite{ABIM}. It was partly motivated by earlier work of Christensen \cite{ghost1}; Bondal and Van den Bergh \cite{BV}; Dwyer, Greenlees, and Iyengar \cite{DGI}; Rouquier \cite{Rouquier}; and Krause and Kussin \cite{KK}.

Upper bounds for projective level are relatively easy to obtain, since any explicit construction of a projective resolution provides such a bound; in particular, $\level{\PP}{F}$ does not exceed the number of non-zero modules in a projective resolution of $F$.  This paper focuses on lower bounds.

In Section \ref{sec:gaps}, we prove the following theorem which gives a lower bound on the projective level of a complex $F$ in terms of the largest gap in the homology of $F$.

\smallskip
\noindent 
{\bf Theorem \ref{th:gaps}.} {\it Let $F$ be a complex of $R$-modules. Assume $\HH{i}{F} = 0$ for all $a < i < b$ with $a,b\in\mathbb{Z}$ and $\HH{0}{ \Omega_{b-1}^R(F)}$ is not projective. Then}
\[\level{\PP}{F} \gs b-a+1.\]
Here, $\Omega_{b-1}^R(F)$ is the $(b-1)$th syzygy of $F$ (see Section \ref{sec:Pre}).

Restricting to commutative Noetherian local rings,  in Section \ref{sec:Inter} we apply Theorem \ref{th:gaps} to deduce a new version of The New Intersection Theorem.   In the following result,  $\level{R}{F}$  is the smallest number of mapping cones needed to build $F$ from finitely generated free $R$-modules.

\smallskip
\noindent 
{\bf Theorem \ref{NewInterThm}.} 
{\it Let $R$ be a Noetherian local ring. Let
\[
\xymatrix{F := 0 \ar[r] & F_{n} \ar[r] & \cdots \ar[r] & F_{0} \ar[r] & 0}
\]
be a complex of finitely generated free $R$-modules such that $\HH{i}{F} $ has finite length for every $i \gs 1$. For any ideal $I$ that annihilates a minimal generator of $\HH{0}{F}$, there is an inequality
\[
\level{R}{F} \gs \dim(R) - \dim(R/I) + 1.
\]}
 This result refines strong forms of the Improved New Intersection Theorem of Evans and Griffith \cite{EG}, due to Bruns and Herzog \cite{BH} and Iyengar \cite{Iyengar}. The proof uses the existence of balanced big Cohen-Macaulay algebras which was recently proved in \cite{A}.   Furthermore, in Theorem \ref{NITDG}, we use Theorem \ref{NewInterThm} to prove that for every commutative Noetherian $R$-algebra $T$ there is an inequality
\[\level{R}{F}\gs \hht IT + 1.\]
This is a special case of a result in \cite{ABIM}.

In Section \ref{sec:Koszul}, we apply results in Section \ref{sec:gaps} to establish lower bounds on the projective level of Koszul complex of an ideal $I$ of a Noetherian local ring $R$.   In particular, we show that the level of the Koszul complex of the maximal ideal in a local ring is one more than the least number of generators of the maximal ideal.  Our motivation for studying the Koszul complex is to understand levels in the context of other better understood invariants.  

The proof of Theorem \ref{th:gaps} and other results in the paper rely on using the vanishing of homology to construct a sequence of ghost maps and then applying the \emph{Ghost Lemma}, Lemma \ref{ghost}.

\section{Preliminaries}\label{sec:Pre}

Let $R$ be an associative ring.  In this paper, complexes of $R$-modules will have the form
\[\xymatrix{F = \cdots \ar[r] & F_{n+1} \ar[r]^-{\partial^F_{n+1}} & F_n \ar[r]^-{\partial^F_{n}} & F_{n-1} \ar[r] & \cdots},\]
with $F_i = 0$ for all $i\ll 0$. Such complexes and their degree 0 chain maps form an abelian category that we denote by $C_+(R)$. We write $X\simeq Y$ if $X$ and $Y$ can be linked by a finite sequence of quasi-isomorphisms; i.e., chain maps that induce isomorphisms in homology. Modules are identified with complexes concentrated in degree 0.

Given a complex $F$, the {\it $k$th suspension of $F$} is the complex $\shift^k F$ with
\[\big( \shift^k F \big)_n = F_{n-k}\qquad \text{and}\qquad \partial^{\shift^k F} = (-1)^k \partial^F.\]
Also, set
\[\ZZ_n(F)=\Ker\partial_n^F\qquad \text{and}\qquad \BB_n(F)=\im\partial_{n+1}^F.\]
For every $i\in \BZ$, let $F_{\gs i}$ and  $F_{\ls i}$ denote the \emph{hard truncations} 
\[\xymatrix{F_{\gs i} := \cdots \ar[r] & F_{i+2} \ar[r]^-{\partial^F_{i+2}} & F_{i+1} \ar[r]^-{\partial_{i+1}^F} & F_i \ar[r] & 0},\]
\[\xymatrix{F_{\ls i} := 0 \ar[r] & F_{i} \ar[r]^-{\partial^F_{i}} & F_{i-1} \ar[r]^-{\partial^F_{i-1}} & F_{i-2} \ar[r] & \cdots},\]
with $F_i$ in homological degree $i$. Note that there is a natural projection 
\[\tau^i\colon F\lra F_{\gs i}\] where $\tau^i_n$ is the identity for $n\gs i$ and zero otherwise.  A projective resolution of $F$ is a quasi-isomorphism $P\lra F$, where $P$ is a complex of projective $R$-modules. Let $[X, Y]$ denote the set of homotopy classes of chain maps from $X\lra Y$ . For a proof of the following result see \cite[5.7.2 and Exercise 5.7.3]{We}.
\begin{proposition}\label{referee} 
Let $F$ be an complex of $R$-modules in $C_+(R)$.
\begin{enumerate}
\item[$(1)$] $F$ has a projective resolution $\pi\colon P\lra F$.
\item[$(2)$] If $\pi'\colon  P'\lra F$ is another projective resolution, then there is a chain map $\rho\colon P'\lra P$ such that $\pi'$ is homotopic to $\pi\rho$. For every complex of $R$-modules $G$, composition with $\rho$ induces a bijective map $[P, G] \to[P',G]$.
\item[$(3)$] If $R$ is left Noetherian and $\HH{i}{F}$ is finitely generated for each $i$, then $P$ can be chosen with $P_i$ finitely generated for each $i$.
\end{enumerate}
\end{proposition}

For a complex $F$, %let $\level{\PP}F$ denote 
let the \emph{projective level} of $F$, denoted by $\level{\PP}F$, be the least integer $n\gs 0$ with the following property: $F$ is quasi-isomorphic to a direct summand of a %some
 complex $P$ with a filtration of %added "a filtration of"
 subcomplexes 
 \[P = P(n)\supseteq\cdots\supseteq P(0)\supseteq P(-1) = 0\]
 such that each $P (i)/P (i -1)$ is a bounded complex of projective $R$-modules with zero differentials.  If no such $n$ exists, let $\level{\PP}F = \infty$. %This invariant is called the \emph{projective level} of $F$.  

A number, $\level{R}F$, is defined by replacing ``projective'' with``finitely generated free module" in the above definition.  This invariant is called the \emph{$R$-level} of $F$.  Quasi-isomorphic complexes have equal levels and the inequality  $\level{\PP}F\ls \level{R}F$ always holds.

A more general definition of level is defined and studied in \cite{ABIM} by using derived categories of $R$-modules.  For a class $\mathcal{X}$ of complexes, a numerical invariant $\altlevel{R}{\mathcal{X}}F$, called the $\mathcal{X}$-level of $F$, was defined in \cite[2.3]{ABIM}.  Roughly speaking, $\altlevel{R}{\mathcal{X}}F$ is the number of steps it takes to build $F$ from elements in $\mathcal{X}$ using mapping cones.  By \cite[4.2]{ABIM}, we have $\altlevel{R}{\{R\}}{F}=\level{R}F$, and the argument in \cite[4.2]{ABIM} shows that $\altlevel{R}{\PP}{F}=\level{\PP}F$ where $\PP$ is the class of projective modules.

%Note that direct summands of finitely generated free modules have level 0 under these definitions of level.  

The following basic facts follow from \cite[2.5.2]{ABIM} and \cite[2.4 (6)]{ABIM}.
\begin{lemma}\label{lemma1}
Let $F$ be a complex of $R$-modules.
\begin{enumerate}
\item[$(1)$] Assume each $F_i$ is projective and $F_i =0$  when $i<a$ or $i>b$ for
some integers $a\ls b$. Then 
\[\level{\PP}F\ls b-a + 1.\]
Furthermore, 
\[\level{R}F\ls b -a + 1\]
  if each $F_i$ is a finitely generated projective module.
\item[$(2)$] If $S$ is a ring and $f\colon C_+(R) \lra C_+(S)$ is an additive functor that preserves quasi-isomorphisms and satisfies $f(R) = S$, we have
\[\level{S}{f(F)}\ls \level{R}F.\]
  \end{enumerate}
  \end{lemma}
Part (1) implies that $\level{\PP}F$ is finite if and only if $F$ is quasi-isomorphic to a bounded complex of projective modules, and $\level{R}F$ is finite if  these modules can be chosen to be finitely generated projective modules;   a complex with the latter property is said to be perfect.

A chain map of complexes $f\colon X \longrightarrow Y$ is said to be \emph{ghost} if the induced  map in homology, $\hh{f} \colon \hh{X} \longrightarrow \hh{Y}$, is zero.  The following result,  known as the \emph{Ghost Lemma}, will be used in the proofs of the main results.

\begin{lemma}[{Ghost Lemma \cite[Theorem 8.3]{ghost1}}]
\label{ghost}
Let $F\lra X$ be a chain map and
\[\xymatrix{X \ar[r] &X(1) \ar[r] &\cdots \ar[r] \ar[r] &X(n)}\]
be a sequence of ghost maps. If the composition $F\lra X(n)$ of these maps is not null-homotopic, then $\level{R}{F} \gs n+1$.
\end{lemma}

Assuming $\HH{i\ll 0}{F}=0$ and choosing a projective resolution $P$ of $F$, we denote the {\it $n$th syzygy module} of $F$ for every $n\in \BZ$ by $\Omega_{n}^R(F):=\shift^{-n}(P_{\gs n})$.   It follows from \cite[1.2]{AIatIll} that the condition $\HH{0}{\Omega_{n}^R(F)}$ is projective does not depend on the choice of $P$.

\section{Gaps and projective levels}\label{sec:gaps}

Let $R$ be an associative ring.  The following is the main result of this section. 
\begin{theorem}\label{th:gaps}
Let $F$ be a complex of $R$-modules. Assume that $\HH{i}{F} = 0$ for all $a < i < b$ with $a,b\in\mathbb{Z}$ and $\HH{0}{ \Omega_{b-1}^R(F)}$ is not projective. Then
\[\level{\PP}{F} \gs b-a+1.\]
\end{theorem}

\begin{proof}
Set $\BB_i(F):=\operatorname{Im}(\partial_{i+1}^F)$ for each $i$, and consider the complex
\[G : = \xymatrix{0 \ar[r] & \BB_{b-1}(F) \ar[r] & F_{b-1} \ar[r] & F_{b-2} \ar[r] & \cdots}.\]
Since $\BB_{b-1}(F)\lra F_{b-1}$ is injective and $\HH{i}G=\HH{i}F$ for all $a<i<b$, we have $\HH{i}G=0$ for all $i>a$.  Therefore each map in the following sequence
\[\xymatrix{G \ar@{->>}[r]  & G_{\gs a+1} \ar@{->>}[r] & G_{\gs a + 2} \ar@{->>}[r] & \cdots \ar@{->>}[r] & G_{\gs b}}\]
 is ghost.
   
By the Ghost Lemma (Lemma \ref{ghost}), it suffices to show that the composed map $\varphi\colon F \lra G \twoheadrightarrow G_{\gs b}$ given by
\[\xymatrix{\cdots \ar[r] & F_{b+1} \ar[r] \ar[d] & F_b \ar@{->>}[d]_-{\varphi_b} \ar[r] & F_{b-1} \ar[r] \ar[d] & \cdots \\ \cdots \ar[r] & 0 \ar[r] & \BB_{b-1}(F)  \ar[r] & 0 \ar[r] & \cdots}\]
 is not null-homotopic.  Assume by way of contradiction that there is a homotopy $\alpha$ from $\varphi$ to 0.  Then the following diagram commutes:
\[\xymatrix{F_b \ar[d]_-{\varphi_b} \ar[r]^{\partial^F_{b}} & F_{b-1} \ar[ld]^-{\alpha_{b-1}} \\ \BB_{b-1}(F)}\]
Since  $\partial^F_b\colon F_b \longrightarrow F_{b-1}$ factors through the inclusion $\BB_{b-1}(F)\lra F_{b-1}$ the existence of such a homotopy $\alpha$ would produce the splitting 
\[\xymatrix{0 \ar[r] & \BB_{b-1}(F) \ar[r] & F_{b-1} \ar[r]  \ar@{-->}@/_1pc/[l]_{\alpha} & F_{b-1}/\BB_{b-1}(F) \ar[r] & 0,}\]
implying $F_{b-1}/\BB_{b-1}(F)$ is projective.  However, 
\[F_{b-1}/\BB_{b-1}(F) = \HH{0}{ \Omega_{b-1}^R(F)}\]
 thus the existence of  $\alpha$ contradicts our assumption that $\HH{0}{ \Omega_{b-1}^R(F)}$ is not projective, proving the claim.
\end{proof}

We can recover the following well known result.  

\begin{corollary}[{\cite[4.5 and 4.6]{ghost1}}]\label{pd}
If $M$ is an $R$-module, then
\[\level{\PP}{M} = \pd (M)+1.\]
\end{corollary}

\begin{proof}
By Lemma \ref{lemma1} (1), we know that $\level{\PP} M\ls \pd M+1$.  Now suppose $\pd M\gs n$.  Let $F$ be a projective resolution of $M$.  By our assumption on $n$, 
\[\HH{0}{\Omega_n F}\cong \Omega_{n-1} M\]
is not projective, hence $\level{\PP}F\gs n+1$.  It follows that 
\[\level{\PP}F\gs \pd M+1.\]  
\end{proof}

\begin{corollary}\label{everyn}
Assume $R$ has infinite global dimension. Then for every $n \gs0$ there exists a complex $F$ with $\level{\PP}{F} = n+1$. 
\end{corollary}

\begin{proof}
When $n=0$, we can take $F=R$. When $n\gs 1$, let $M$ be an $R$-module such that $\pd(M)\gs n$. Let $P$ be a projective resolution of $M$, and set $F=P_{\ls n}$. By Lemma \ref{lemma1} (1), $\level{R}{F}\leqslant n+1$. From $\HH{0}{\Omega_{n-1}^R(F) }=\Omega_{n-1} M$ and $\pd(M)\gs n$, we see that $\HH{0}{\Omega_{n-1}^R(F) }$ is not projective.  Since $\HH{i}{F}=0$ for $0<i<n$, Theorem \ref{th:gaps} gives the desired equality.  
\end{proof}

The following result will be used several times in Section \ref{sec:Koszul}.

\begin{proposition}\label{gapsmap}
Let $F$ be a complex of $R$-modules and $I\subseteq R$ a left ideal such that $\Image(\partial^F_i)\subseteq IF_{i-1}$ for every $i$. Assume there exists a chain map $\eta\colon  F \longrightarrow G$ where $G$ is a complex such that $\HH{i}{G} = 0$ for all $a < i < b$, and 
\[\Image(\eta_b) \not\subseteq I G_b +\ZZ_b(G).\]
 Then
\[\level{\PP}{F} \gs b-a+1.\]
\end{proposition}

\begin{proof}
The proof is similar to that of Theorem \ref{th:gaps}. 
Consider the complex 
\[G' : = \xymatrix{0 \ar[r] & G_{b}/ \ZZ_b(G) \ar[r] & G_{b-1} \ar[r] & \cdots}\]
and let $\eta' \colon F \longrightarrow G'$ be the composition of $\eta$ and the natural chain map from $G$ to $G'$. Since $\HH{i}{G'}=0$ for all $i>a$, the $b-a$ maps in the sequence 
\[\xymatrix{G' \ar@{->>}[r] & G'_{\gs a+1} \ar@{->>}[r] & G'_{\gs a+2} \ar@{->>}[r] & \cdots \ar@{->>}[r] & G'_{\gs b}}\]
are ghost.  Let $\varphi\colon F\lra G'_{\gs b}$ be the composition of these maps.  The existence of a homotopy $\alpha$ from $\varphi$ to 0 would imply
\[\im(\eta_b)=\im(\varphi_b)=\alpha_{b-1}\partial_b^{F}(F_b)\subseteq IG'_b.\]
Since $\im(\partial^F_i)\subseteq IF_{i-1}$ for every $i$, the existence of a chain homotopy from $\varphi$ to the zero map implies $\Image(\varphi_b)=\Image(\eta'_b)\subseteq IG'_b$.  But this implies that $\Image(\eta_b)$ is contained in $IG_b +\ZZ_b(G)$, contradicting the assumptions. Therefore, $\varphi$ is not null-homotopic and the result follows from Lemma \ref{ghost}, the Ghost Lemma.  
\end{proof}

\begin{remark}\label{tacos}
Since $\level{\PP}{F}\ls \level{R}{F} $, projective level can be replaced with $R$-level in the conclusions of Theorem \ref{th:gaps} and Proposition \ref{gapsmap}.  Therefore, using finitely generated modules and the full strength of Lemma \ref{lemma1} (1), a similar replacement works for Corollary \ref{pd} and  Corollary \ref{everyn} when $R$ is left Noetherian.
%When $F$ is a perfect complex the conclusions of Theorem \ref{th:gaps} and Corollary \ref{everyn} hold with $\level{\PP}F$ replaced by $\level{R}F$. A similar replacement works for Corollary \ref{pd} if $R$ is left noetherian and $M$ is finitely generated. The only change needed in the proofs is to use resolutions by finitely generated projective modules.
\end{remark}

\section{The Improved New Intersection Theorem for levels}\label{sec:Inter}

Let $R$ be a commutative Noetherian local ring with maximal ideal $\fm$. The results in this section use the existence of balanced big Cohen-Macaulay algebras.  Recall that an $R$-algebra $S$ is said to be a {\it balanced big Cohen-Macaulay algebra} if every system of parameters for $R$ forms an $S$-regular sequence. Until recently it was only known that balanced big Cohen-Macaulay algebras exist when $R$ is equicharacteristic or $\dim R\ls 3$: see  \cite[Chapter 8]{BH}, \cite[Theorem 8.1]{HH2}, \cite[Theorem 1]{HH}, and \cite{H2}.  However, in his recent work on the Direct Summand Conjecture,  Andr\'{e} shows in \cite{A} that they exist for all rings.

\begin{theorem}\label{NewInterThm}
%Assume $R$ has a balanced big Cohen-Macaulay algebra. 
Let
\[
\xymatrix{F := 0 \ar[r] & F_{n} \ar[r] & \cdots \ar[r] & F_{0} \ar[r] & 0}
\]
be a complex of finitely generated free $R$-modules such that $\HH{0}{F}\ne 0$ and $\HH{i}{F}$ has finite length for every $i \gs 1$. If $I$ is an ideal that annihilates a minimal generator of $\HH{0}{F}$, there is an inequality
\[
\level{R}{F} \gs \dim(R) - \dim(R/I) + 1.
\]
\end{theorem}

\begin{proof}
It suffices to consider the case when $\dim(R)-\dim(R/I)\gs 1$. Also, replacing $F$ by its minimal free resolution, we  can assume that $F$ is minimal and that $F_n\ne 0$.  By \cite[0.7.1.]{A}, there exists a balanced big Cohen-Macaulay $R$-algebra $S$.
% that is say, $d(F)\subseteq \fm F$, where $\fm$ is the maximal ideal of $R$; \red{minimal complexes defined before}; see~\cite[2.4]{Roberts}. 
Set $G:=F\otimes_R S$ and 
\[
s:= \sup\{i\mid \HH{i}{G}\neq 0\}\,.
\]
Now, the proof of \cite[3.1]{Iyengar} yields an inequality
\begin{equation}
\label{eqnNIT}
s\ls n - \dim (R) + \dim (R/I)\,.
\end{equation}
In particular, $s\ls n-1$.

Set $\Omega:=\HH{0}{\Omega_{n-1}^S(G)}$.  We claim that the $S$-module $\Omega$ is not projective.  Indeed, since $s\ls n-1$ the complex $0\lra G_{n}\lra G_{n-1}\lra 0$, with $G_{n-1}$ in degree $0$, is a free resolution of $\Omega$. Since $S$ is a big Cohen-Macaulay algebra we have $\fm S\ne S$ where $\fm$ is the maximal ideal of $R$. Hence, there exists a maximal ideal $\fn$ of $S$ containing $\fm S$. By the minimality of $F$, we have   
\[
\Tor 1S{S/\fn}{\Omega} \cong (S/\fn)\otimes_S G_{n} \cong (S/\fn)\otimes_R F_n\ne 0.
\]
This implies that $\Omega$ is not flat, which justifies the claim.

Given the preceding claim, Theorem~\ref{th:gaps} and Remark \ref{tacos} yield the second inequality below:
\[
 \level{R}F \gs \level{S}{G} \gs n - s +1.
\]
The first inequality is Lemma \ref{lemma1} (2) applied to the functor $f=-\otimes_R S$. It remains to recall \eqref{eqnNIT}.
\end{proof}

This result can be used to recover a special case of the New Intersection Theorem in \cite{ABIM}. Recall that
\[\superheight{I}=\sup\{\hht{IT}\mid T\mbox{ is a Noetherian }R\mbox{-algebra}\}.\]

\begin{theorem}[{\cite[5.1]{ABIM}}]\label{NITDG}
Let $F$ be a perfect complex and let $I\subset R$ be the annihilator of $\bigoplus_{i\in\mathbb{Z}}\HH{i}F$.  Then we have
\[\level{R}{F}\gs \superheight{I}+1.\]
\end{theorem}

\begin{proof}
Let $R\lra T$ be a Noetherian $R$-algebra, and let $\p\in\Spec{T}$ be minimal over $IT$.  %Since $R$ is equicharacteristic, so are $T$ and $T_\p$.  Therefore, by \cite[Theorem 8.1]{HH2} (see also \cite[Theorem 1]{HH}), $T_\p$ has a balanced big Cohen-Macaulay algebra.  
The ideal $IT_\p$ is the annihilator of $\oplus_{i\in\mathbb{Z}}\HH{i}{F\otimes T_\p}$, hence $\HH{i}{F\otimes T_\p}$ has finite length for every $i$. From Lemma \ref{lemma1} (2)  and Theorem \ref{NewInterThm} we obtain
\[\level{R}{F}\gs \level{T_\p}{T_\p\otimes_R F}\gs \dim T_\p-\dim T_\p/IT_\p+1=\hht{\p}+1.\]
The result follows.  
\end{proof}

\begin{remark}

By Lemma \ref{lemma1} (1), this result recovers the classical form of the New Intersection Theorem.

\end{remark}

\section{Koszul complexes over local rings}\label{sec:Koszul}

Let $(R, \fm, k)$ be a commutative Noetherian local ring. As usual, $\dim(R)$ will denote the Krull dimension of $R$, $\beta(I)$ will be the minimal number of generators of an ideal $I\subseteq R$, and $\edim(R)$ will be $\beta(\fm)$.  Given an ideal $I$ in a ring $R$, we let $K(I)$ denote the Koszul complex on some set of generators of $I$, which we will show to be independent of the choice of generators. The goal of this section is to find estimates for the $R$-level of $K(I)$. 

\begin{lemma}
For an ideal $I\subseteq R$, the number $\level{R}{K(I)}$ is well defined, i.e.~ if $x_1,\dots,x_n$ and $y_1,\dots, y_m$ are both generating sets of $I$, then 
\[\level{R}{K(x_1,\dots,x_n)}=\level{R}{K(y_1,\dots,y_m)}\]
where $K(x_1,\dots,x_n)$ is the Koszul complex on $x_1,\dots,x_n$ and $K(y_1,\dots,y_m)$ is similarly defined. 
\end{lemma}

\begin{proof}
It suffices to show that if $y\in (x_1,\dots, x_n)$ then 
\[\level{R}{K(x_1,\dots,x_n)}=\level{R}{K(x_1,\dots,x_n,y)}.\]
Since $K(x_1,\dots,x_n,y)$ is the mapping cone of the multiplication map
\[\xymatrix{K(x_1,\dots,x_n) \ar[r]^{y} & K(x_1,\dots,x_n)}\]
it is easy to see that multiplication by $x_i$ on $K(x_1,\dots,x_n)$ is homotopic to 0, and thus the same is true for multiplication by $y$ since $y\in (x_1,\dots,x_n)$.  Therefore, $K(x_1,\dots,x_n,y)$ is quasi-isomorphic to 
\[K(x_1,\dots,x_n)\oplus \Sigma K(x_1,\dots,x_n)\]
where $\Sigma$ denotes the shift functor.  Hence, the result follows from \cite[2.4 (1) and (3)]{ABIM}.  
\end{proof} 

\begin{theorem}\label{th:Koszul}\  
\begin{enumerate}
\item[$(1)$] There is an equality 
\[\level{R}{K(\fm)} = \edim(R) + 1.\]
\item[$(2)$] If I is an ideal generated by a system of parameters, then
\[\level{R}{K(I)} = \dim(R) + 1.\]
\item[$(3)$] For any ideal $I\subseteq R$,
\[\level{R}{K(I)}\gs \depth(I,R)\]
where $\depth(I,R)$ is the length of the longest $R$-regular sequence in $I$.  
\item[$(4)$] If $I$ is an ideal generated by a regular sequence and $R$ is equicharacteristic, then for every integer $c\gs 1$ there is an equality 
\[\level{R}{K(I^c)}=\beta(I)+1.\]
\end{enumerate}
\end{theorem}

The proof of the theorem requires several lemmas.  Recall that $\Tor{}{R}{R/I}{k}$ is a graded commutative $k$-algebra see \cite[2.3.2]{Avramov6Lectures}. Therefore, there is a natural map of $k$-algebras 
\[\kappa^I\colon \bigwedge \Tor{1}{R}{R/I}{k}\lra \Tor{}{R}{R/I}{k}.\]

\begin{lemma}\label{kappa} For an ideal $I$, we have an equality
\[\level{R}{K(I)} \gs \sup\{b\mid \kappa^I_b\neq 0\}+1.\]
\end{lemma}
\begin{proof}
Let $G$ be a minimal free resolution of $R/I$.  The map $K(I)\lra R/I$ extends to a map of complexes $\eta\colon K(I)\lra G$.  Notice $K(I)\otimes_R k \simeq \bigwedge \Tor{1}{R}{R/I}{k}$ and $G\otimes_R k\simeq  \Tor{}{R}{R/I}{k}$ as graded $k$-algebras, and $\kappa^I=\eta\otimes_R k$.  Therefore, $\kappa^I_b \neq 0$ is equivalent to  $\Image(\eta_b)\not\subseteq \fm G_b$. But since $\HH{b}{G}=0$, we have 
\[\ZZ_b(G)=\BB_b(G)\subseteq \fm G_b\] 
which implies $\Image(\eta_b)\not\subseteq \fm G_b+\ZZ_b(G)$.  The conclusion now follows from Proposition \ref{gapsmap}.
\end{proof}

\begin{proof}[{Proof of Theorem \ref{th:Koszul} (1)}]
Let $G$ be a minimal resolution of $k$. By Lemma \ref{lemma1} (1), we only need to prove $\level{R}{K(\fm)}\gs \edim(R)+1$.  The map 
\[\kappa^{\fm}\colon \bigwedge\Tor{1}{R}{k}{k}\lra \Tor{}{R}{k}{k}\] 
is an injection by \cite[Lemme 3 (b)]{S2} or alternatively \cite[Theorem 7]{Tate}.   Since the highest degree of $\bigwedge\Tor{1}{R}{k}{k}$ is $\edim(R)$, the map $\kappa^{\fm}_{\edim(R)}$ is non-zero.  Therefore, the result follows from Lemma \ref{kappa}.
\end{proof}

Theorem \ref{th:Koszul} (2) follows from the recently proved Canonical Element Conjecture.  The Canonical Element Conjecture was introduced in \cite{H} where it was shown to be equivalent to the Direct Summand Conjecture and implied by the existence of big Cohen-Macaulay modules.  %Hence these conjectures were known to be true when $R$ is equicharacteristic case or $\dim R\le 3$; see \cite{H} and \cite{He}.  
Recently the Direct Summand Conjecture was proved for all Noetherian rings in \cite{A} and \cite{B}.

\begin{proof}[{Proof of Theorem \ref{th:Koszul} (2)}]
Set $d = \dim R$ and $I = (x_1,...,x_d)$ where $x_1, ..., x_d$ is a system or parameters for $R$. Lemma \ref{lemma1} (1) implies $\level{R}{K(I)}\ls \dim(R) +1$. We now show the reverse inequality. Let $G$ be a minimal free resolution of $k$ and $\phi\colon K(I)\lra G$ be a lifting of the quotient map $R/I\lra k$. The claim follows from an equivalent form of The Canonical Element Conjecture (see \cite[page 505, (6)]{H}) saying that the induced map 
\[\phi_d\colon  K_d(I) \lra \frac{\Omega_d(k)}{I\Omega_d(k)}\cong\frac{G_{d}}{IG_{d}+Z_{d}(G)}\]
is nonzero.  The result now follows from Proposition \ref{gapsmap}.
\end{proof}

Recall that a complex $F$ is said to be \emph{minimal} if it is a complex of finitely generated free modules and $\partial_i(F_i) \subseteq \fm F_{i-1}$ for every $i$. 

\begin{lemma}\label{indecomposable}
Let $F$ be a minimal complex. Suppose $\HH{i}{F}=0$ for all $a<i<b$ with $a,b\in\mathbb{Z}$. If $\partial_b\ne 0$, then 
\[\level{R}F\gs b-a + 1.\]  
\end{lemma}

\begin{proof}

Let $\eta\colon  F\lra F$ be the identity map.  Since $\partial_b^F$ is not zero, $\ZZ_b(F_b)\ne F_b$.  Therefore, $\im\eta_b=F_b\not\subseteq \fm F_b+\ZZ_b(F)$ by Nakayama's Lemma.  The result now follows from Proposition \ref{gapsmap}.
\end{proof}

\begin{proof}[{Proof of Theorem \ref{th:Koszul} (3)}]
Set $n=\beta(I)$. Clearly $\partial_n^{K(I)}$ is not zero.   Since $\HH{i}{K(I)}=0$ for all $n-\depth(I,R)<i\ls n$, Lemma \ref{indecomposable} implies that
\[\level{R}{K(I)}\gs n-(n-\depth(I,R))+1=\depth(I,R) +1.\]
\end{proof}

\begin{remark}

In the equicharacteristic case, Theorem \ref{th:Koszul} (3) also follows from Theorem \ref{NITDG} by the inequality 
\[\superheight I\gs \hht I\gs \depth(I,R).\]

\end{remark}

\begin{lemma}\label{powers}
If $R$ is equicharacteristic, then for any proper ideal $I\subset R$ and any $c\in\mathbb{N}$ there is an inequality
\[\level{R}{K(I^c)} \ls \beta(I) + 1.\]
\end{lemma}
\begin{proof}
Since $R$ is equicharacteristic, it contains the residue field $k$.  Let $r_1, \ldots, r_n$ be a minimal generating set of $I$. Let $S = k \left[ x_1, \ldots, x_n \right]$ be a polynomial ring, set $\mathfrak{n} = (x_1, \ldots, x_n)$, and consider the map $f\colon S \longrightarrow R$ defined by $f(x_i) = r_i$. Note that 
\[K(I^c,R) = K(\mathfrak{n}^c,S) \otimes_S R=K(\mathfrak{n}^c,S) \dtensor{S} R.\] Hence, by Lemma \ref{lemma1} (2)  we have
\[\level{R}{K(I^c, R)} \ls \level{S}{K(\mathfrak{n}^c, S)}=n+1,\]
where the last equality holds by \cite[5.3]{ABIM}. The result follows.
\end{proof}

\begin{proof}[{Proof of Theorem \ref{th:Koszul} (4)}]
The statement follows from Theorem \ref{th:Koszul} (3) and Lemma \ref{powers}.  Indeed, these facts yield
\[n+1\gs \level{R}{K(I^c)}\gs \depth(I^c,R) +1=\depth(I,R) +1=n+1.\]
\end{proof}

A complex $F$ is said to be {\it indecomposable in the derived category} if $\HH{}{F}\ne 0$ and given any quasi-isomorphism $F \simeq A \oplus B$ with $A$ and $B$ complexes, then either $\HH{}{A}\simeq 0$ or $\HH{}{B} \simeq 0$.  
\begin{proposition}\label{KosInd}
Suppose $(R,\fm,k)$ is local.  For any proper ideal $I\subseteq R$, the complex $K(I)$ is  indecomposable in the derived category.
\end{proposition}

\begin{proof}
Pick a minimal generating set $x_1,\ldots,x_n$ of $I$, and set $K:=K(x_1,\dots,x_n)$. By \cite[Chap. IV, App. I, \S 1, Proposition 3]{S}, the cycles of $K$, and hence the boundaries of $K$, lie in $\fm K$, and so $K$ is minimal.  Suppose that $K \simeq A \oplus B$ where $A$ and $B$ are perfect complexes.  Then $A$ and $B$ have minimal free resolutions, and so we may assume that $A\oplus B$ is minimal.  

Any quasi-isomorphism of bounded minimal complexes of free modules is an isomorphism (see \cite[1.1.2]{Avramov6Lectures}), hence $R= K_0 = A_0 \oplus B_0$. Since $R$ is indecomposable as an $R$-module, we must have $A_0 = 0$ or $B_0 = 0$. 

Suppose without loss of generality that $B_0 = 0$.  Since $\HH{}{B}\ne 0$, we can consider $a = \min \lbrace i | B_i \neq 0 \rbrace$. Since  $\Image(\partial^B_a)=0$, there exists $v \in K_a$ such that $v \notin \fm K_a$ and $\partial^K_a(v) = 0$. However, as mentioned earlier, the cycles of $K$ are contained in $\fm K$.  This gives a contradiction, implying that $\HH{}{B}=0$.  
\end{proof}

\begin{remark}

Theorem \ref{th:Koszul} (4) and Proposition \ref{KosInd} show that the difference between the length and the $R$-level of a minimal, derived indecomposable perfect complex can be arbitrarily long.

\end{remark}

\section*{Acknowledgements} This project was started at the Mathematical Research Communities 2015 in Commutative Algebra organized by S.~B.~Iyengar, L.~Sega, K.~Schwede, G.~Smith, and W. ~Zhang. We thank the MRC and the American Mathematical Society for their support and for providing us with such a stimulating environment. We are especially grateful to S.~B.~Iyengar for introducing us to this wonderful topic and advising us during the writing of this article.

We would also like to thank the referee whose thorough report greatly improved this article.

\end{document}